\documentclass[12pt]{amsart}
\usepackage{amscd,amsmath,amsthm,amssymb,enumerate}
\usepackage[left]{lineno}
\usepackage{pstricks}
\usepackage[utf8]{inputenc}
\usepackage{epstopdf}

\usepackage{comment}

\usepackage[all]{xy}

 \def\G{{\mathcal G}}

 \def\opn#1#2{\def#1{\operatorname{#2}}} 
 \opn\chara{char} \opn\length{\ell} \opn\pd{pd} \opn\rk{rk}
 \opn\projdim{proj\,dim} \opn\injdim{inj\,dim} \opn\rank{rank}
 \opn\depth{depth} \opn\grade{grade} \opn\height{height}
 \opn\embdim{emb\,dim} \opn\codim{codim}
 
 \opn\Tr{Tr} \opn\bigrank{big\,rank}
 \opn\superheight{superheight}\opn\lcm{lcm}
 \opn\trdeg{tr\,deg}
 \opn\reg{reg} \opn\lreg{lreg} \opn\ini{in} \opn\lpd{lpd}
 \opn\size{size} \opn\sdepth{sdepth}
 \opn\link{link}\opn\fdepth{fdepth}\opn\lex{lex}

 \opn\div{div} \opn\Div{Div} \opn\cl{cl} \opn\Cl{Cl}
 \opn\Spec{Spec} \opn\Supp{Supp} \opn\supp{supp} \opn\Sing{Sing}
 \opn\Ass{Ass} \opn\Min{Min}\opn\Mon{Mon}
 \opn\Ann{Ann} \opn\Rad{Rad} \opn\Soc{Soc} \opn\ir{ir}
 \opn\Im{Im} \opn\Ker{Ker} \opn\Coker{Coker} \opn\Am{Am}
 \opn\Hom{Hom} \opn\Tor{Tor} \opn\Ext{Ext} \opn\End{End}
 \opn\Aut{Aut} \opn\id{id}
 
 \opn\nat{nat}
 \opn\pff{pf}
 \opn\Pf{Pf} \opn\GL{GL} \opn\SL{SL} \opn\mod{mod} \opn\ord{ord}
 \opn\Gin{Gin} \opn\Hilb{Hilb}\opn\sort{sort}
 \opn\aff{aff} \opn
 \con{conv} \opn\relint{relint} \opn\st{st}
 \opn\lk{lk} \opn\cn{cn} \opn\core{core} \opn\vol{vol}  \opn\inp{inp} \opn\nilpot{nilpot}
 \opn\link{link} \opn\star{star}\opn\lex{lex}\opn\set{set}
 \opn\width{wd}
 \opn\ecart{ecart}
 \opn\gr{gr}
 
 \def\pot#1#2{#1[\kern-0.28ex[#2]\kern-0.28ex]}

 \opn\dirlim{\underrightarrow{\lim}}
 \opn\inivlim{\underleftarrow{\lim}}
 \def\Implies{\ifmmode\Longrightarrow \else
         \unskip${}\Longrightarrow{}$\ignorespaces\fi}
 \def\implies{\ifmmode\Rightarrow \else
         \unskip${}\Rightarrow{}$\ignorespaces\fi}
 \def\iff{\ifmmode\Longleftrightarrow \else
         \unskip${}\Longleftrightarrow{}$\ignorespaces\fi}

 \let\:=\colon

 \def\Soc{{\mathbf Soc}}

 \def\opn#1#2{\def#1{\operatorname{#2}}} 
 \opn\chara{char} \opn\length{\ell} \opn\pd{pd} \opn\rk{rk}
 \opn\projdim{proj\,dim} \opn\injdim{inj\,dim} \opn\rank{rank}
 \opn\depth{depth} \opn\grade{grade} \opn\height{height}
 \opn\bigheight{bigheight}
 \opn\embdim{emb\,dim} \opn\codim{codim}
 
 \opn\superheight{superheight}\opn\lcm{lcm}
 \opn\trdeg{tr\,deg}
 \opn\reg{reg} \opn\lreg{lreg} \opn\ini{in} \opn\lpd{lpd}
 \opn\size{size} \opn\sdepth{sdepth}
 \opn\link{link}\opn\fdepth{fdepth}\opn\lex{lex}
 \opn\type{type}
 \opn\gap{gap}
 \opn\arithdeg{arith-deg}
 \opn\Deg{Deg}
 \opn\sat{sat}
 \opn\mat{mat}
 \opn\Mat{Mat}
 \opn\div{div} \opn\Div{Div} \opn\cl{cl} \opn\Cl{Cl}
 \opn\Spec{Spec} \opn\Supp{Supp} \opn\supp{supp} \opn\Sing{Sing}
 \opn\Ass{Ass} \opn\Min{Min}\opn\Mon{Mon} \opn\Max{Max}
 \opn\Ann{Ann} \opn\Rad{Rad} \opn\Soc{Soc}
 \opn\Im{Im} \opn\Ker{Ker} \opn\Coker{Coker} \opn\Am{Am}
 \opn\Hom{Hom} \opn\Tor{Tor} \opn\Ext{Ext} \opn\End{End}
 \opn\Aut{Aut} \opn\id{id}
 
 \opn\nat{nat}
 \opn\pff{pf}
 \opn\Pf{Pf} \opn\GL{GL} \opn\SL{SL} \opn\mod{mod} \opn\ord{ord}
 \opn\Gin{Gin} \opn\Hilb{Hilb}\opn\sort{sort}
 \opn\PF{PF}\opn\Ap{Ap}
 \opn\mult{mult}
 \opn\bight{bight}
  \opn\ir{ir}
 \opn\aff{aff}
 \opn\relint{relint} \opn\st{st}
 \opn\lk{lk} \opn\cn{cn} \opn\core{core} \opn\vol{vol}  \opn\inp{inp} \opn\nilpot{nilpot}
 \opn\link{link} \opn\star{star}\opn\lex{lex}\opn\set{set}
 \opn\width{wd}
 \opn\Fr{F}
 \opn\QF{QF}
 \opn\G{G}
 \opn\type{type}\opn\res{res}
 \opn\conv{conv}
 \opn\Shad{Shad}
 \opn\gr{gr}
 
 \def\pot#1#2{#1[\kern-0.28ex[#2]\kern-0.28ex]}

 \opn\dirlim{\underrightarrow{\lim}}
 \opn\inivlim{\underleftarrow{\lim}}
 \def\Implies{\ifmmode\Longrightarrow \else
         \unskip${}\Longrightarrow{}$\ignorespaces\fi}
 \def\implies{\ifmmode\Rightarrow \else
         \unskip${}\Rightarrow{}$\ignorespaces\fi}
 \def\iff{\ifmmode\Longleftrightarrow \else
         \unskip${}\Longleftrightarrow{}$\ignorespaces\fi}

 \let\:=\colon
\theoremstyle{plain}
\newtheorem{theorem}{Theorem}[section]

\newtheorem{corollary}[theorem]{Corollary}

\newtheorem{lemma}[theorem]{Lemma}

\theoremstyle{definition}

\newtheorem{example}[theorem]{Example}

\newtheorem{remark}[theorem]{Remark}

 \let\epsilon\varepsilon
 \let\kappa=\varkappa
 \def\qed{\ifhmode\textqed\fi
       \ifmmode\ifinner\quad\qedsymbol\else\dispqed\fi\fi}
 \def\textqed{\unskip\nobreak\penalty50
        \hskip2em\hbox{}\nobreak\hfil\qedsymbol
        \parfillskip=0pt \finalhyphendemerits=0}
 \def\dispqed{\rlap{\qquad\qedsymbol}}
 
 \opn\dis{dis}
 \def\pnt{{\raise0.5mm\hbox{\large\bf.}}}
 
 \opn\Lex{Lex}

\newcommand{\rmf}{\mathrm{f}}

\newcommand{\m}{\mathfrak{m}}

\newcommand{\q}{\mathfrak{q}}

\title{Irreducible multiplicity of idealizations}

\author{Tran Nguyen An}
\address{Tran Nguyen An: Thai Nguyen University of Education, Vietnam}
\email{antn@tnue.edu.vn}

\thanks{2020 {\em Mathematics Subject Classification.} 13A15, 13H10, 13H15}
\thanks{{\em Key words and phrases.} index of reducibility, irreducible multiplicity, dimension of socle, Cohen-Macaulay type, idealization}
\thanks{This work is supported by the Vietnam Ministry of Education and Training  under grant number B2025-TNA-03}

\begin{document}

\begin{abstract} Let $(R,\mathfrak{m})$ be a Noetherian local ring and $M$ a finitely generated $R$-module. We study the relations of the index of reducibility and the irreducible multiplicity of an $\mathfrak{m}$-primary ideal of $R$ and these of $\mathfrak{m} \times M$-primary ideal of the idealization. This generalizes one of the main results of S.Goto et al. (see \cite[Theorem 2.2]{GKL}).
\end{abstract}
\maketitle

\section{Introduction}
 Let $(R,\m)$ be a Noetherian local ring of dimension $d$ and $I$ an $\m$-primary ideal of $R$.  Let $M$ be a nonzero finitely generated $R$-module of dimension $t$. Recall that a submodule $N$ of $M$ is called an {\it irreducible submodule} if $N$ cannot be written as an intersection of two properly larger submodules of $M$.  The number of irreducible components of an irredundant irreducible decomposition of $N$, which is independent of the choice of the decomposition by N. Noether \cite{Noe}, is called the {\it index of reducibility} of $N$ and denoted by $\mathrm{ir}_M(N)$.
 We denote by $\ell_R(*)$ the length of an $R$-module $*$. Consider  $\ell_R(I^{n+1}M:_M \m /I^{n+1}M)$, the dimension of the socle of $M/I^{n+1}M$. Since $I$ is $\m$-primary, we also have 
$$\ell_R(I^{n+1}M:_M \m/I^{n+1}M)=\ir_M(I^{n+1}M),$$
where $\ir_M(I^{n+1}M)$ is the index of reducibility of $I^{n+1}M$.  In 2015, N. T. Cuong, P. H. Quy, and H. L. Truong proved that the function $\mathrm{ir}_M(I^{n+1}M)$ agrees with a polynomial function of degree $t-1$ for $n\gg 0$ (\cite[Theorem 4.1]{CQT}). That is, there exist integers $\rmf_I^0(M), \dots, \rmf_I^{t-1}(M)$ such that
\begin{align*} \label{eq00}
	\begin{split} 
		\mathrm{ir}_M(I^{n+1}M) =& \ell_R(I^{n+1}M:_M \m/I^{n+1}M) \\
		= & \rmf_I^0(M)\binom{n+t-1}{t-1}-\rmf_I^1(M)\binom{n+t-2}{t-2}+ \cdots +(-1)^{t-1}\rmf_I^{t-1}(M)
	\end{split}
\end{align*}
for $n\gg 0$. 
The numbers $\rmf_I^0(M), \dots, \rmf_I^{t-1}(M)$ are called the {\it irreducibility coefficients} of $M$ with respect to $I$ and the leading coefficient $\rmf_I^0(M)$ is called the {\it irreducible multiplicity} of $M$ with respect to $I$ (see \cite{AK}, \cite{T3}). The index of reducibility and irreducible multiplicity have a strong connection with the structure of rings (see \cite{AK}, \cite{ADKN}, \cite{CQT}, \cite{EN}, \cite{GSu}, \cite{Q2}, \cite{T3}). Some relations between (Hilbert-Samuel) multiplicity and irreducible multiplicity are also given (see \cite{AK}, \cite{T3}). Irreducible decomposition and index of reducibility of homogeneous ideal in idealization of a module are given in \cite{An}.

 The purpose of this paper is to study the relations of the index of reducibility and the irreducible multiplicity of an $\mathfrak{m}$-primary ideal of $R$ and these of $\mathfrak{m} \times M$-primary ideal of the idealization. Recall that the Cartesian product $R\times M$ is a commutative ring concerning componentwise addition and multiplication defined by
$$ (r_1, m_1)(r_2, m_2)=(r_1 r_2, r_1 m_2+r_2m_1), $$
where $r_1, r_2\in R$ and $m_1, m_2\in M$. This commutative ring is called the {\it idealization} of $M$ or the {\it trivial extension} of $R$ by $M$, denoted by $R\ltimes M$. Idealization of a module was introduced by Nagata \cite{Na}. One can find many properties of idealizations in \cite{AW} and \cite{Na}. 

The main result of this paper is as follows.

\begin{theorem}\label{T2} Let $I$ be an $\m$-primary ideal of $R$. Set $J=I\times IM$. Then
	\begin{equation*}
		\ir_M(I^{n+1}M) \leq \ir_{R\ltimes M}(J^{n+1}) \leq \ir_R(I^{n+1})+\ir_M(I^{n+1}M)
	\end{equation*}
for all $n\ge 0$. In particular, we have the following. 
	$$
	\mathrm{f}_I^0(M) \leq \mathrm{f}_J^0(R\ltimes M) \leq \mathrm{f}_I^0(R)+\mathrm{f}_I^0(M) .
	$$
	For $n\ge 0$, we further have the following.
	\begin{enumerate}[\rm(i)] 
	\item $\ir_M(I^{n+1}M) = \ir_{R\ltimes M}(J^{n+1}) $ if and only if $M/I^{n+1}M$ is a faithful $R/I^{n+1}$-module.
	\item $\ir_{R\ltimes M}(J^{n+1}) = \ir_R(I^{n+1})+\ir_M(I^{n+1}M)$ if and only if $(I^{n+1}:_R \m)M= I^{n+1}M$. 
	\end{enumerate}
\end{theorem}

In \cite{GKL}, S. Goto, S. Kumashiro, and N. T. H. Loan explored the Cohen-Macaulay type of idealization. Suppose $M$ is a Cohen-Macaulay $R$-module and $Q$ is a parameter ideal of $M$, then $\mathrm{f}_Q^0(M)$ is the Cohen-Macaulay type of $M$ (see Remark \ref{lhvb}). With this observation, Theorem \ref{T2} is a generalization of \cite[Theorem 2.2]{GKL}. 

Note that the inequalities in Theorem \ref{T2} can be strict (Example \ref{ineq}). We also remark that irreducible multiplicity is not compatible with the reduction ideal (Example \ref{VD2}).

In the next section we prove the main results.

\section{Proof of main results}

Let $(R,\m)$ be a Noetherian local ring. Then $R\ltimes M$ is also a Noetherian local ring with unique maximal ideal $\m \times M$, and $\dim R\ltimes M=\dim R$ (see \cite{AW})). 
Let 
\begin{center}
$\rho: R\ltimes M \rightarrow R; (a, m) \mapsto a$ \quad and \quad $\sigma: R\rightarrow R\ltimes M; a \mapsto (a,0)$
\end{center}
be the canonical projection and the canonical inclusion, respectively. Then $\rho$ and $\sigma$ are local ring homomorphisms. Passing to these morphisms, an $R$-module $L$ can be regarded as an $R\ltimes M$-module, and vice versa. 
We then note that the length of $L$ as an $R$-module and that of $L$ as an $R\ltimes M$-module are the same, that is, $\ell_R(L)=\ell_{R\ltimes M} (L)$. 

To prove Theorem \ref{T2}, we need the following lemma about colon ideals in idealization. 

\begin{lemma}\label{colonofidealization} {\rm (\cite{Ali})}
Let $I\times N$ and $I'\times N'$ be homogeneous ideals of $R\ltimes M$. Then
	$$(I\times N) :_{R\ltimes M} (I'\times N')=((I:_RI') \cap (N:_RN'))\times (N:_MI').$$
	Furthermore $(I\times N) :_{R\ltimes M} (I'\times N')$ is a homogeneous ideal of $R\ltimes M$.
	 \end{lemma}

\begin{proof}[Proof of Theorem \ref{T2}] By \cite[Theorem 3.3 (2)]{AW} and by induction on $n$ we have $$J^{n+1}=(I\times IM)^{n+1}=I^{n+1}\times I^{n+1}M.$$
	By Lemma \ref{colonofidealization}, 
	$$J^{n+1}:_{R \ltimes M} (\m \times M)=((I^{n+1}:_R \m)\cap (I^{n+1}M:_R M))\times (I^{n+1}M:_M\m).$$ Then
	\begin{align*}
		\Big(J^{n+1} &:_{R \ltimes M} (\m \times M)\Big)/J^{n+1} \\
		&=\Big(((I^{n+1} :_R \m)\cap (I^{n+1}M:_R M))\times (I^{n+1}M:_M\m)\Big)/\Big(I^{n+1}\times I^{n+1}M\Big)\\
		&\cong \Big(\big((I^{n+1}:_R \m)\cap (I^{n+1}M:_R M)\big)/I^{n+1}\Big)\times \Big((I^{n+1}M:_M\m)/I^{n+1} M\Big).
		\end{align*}
	So,
	\begin{align*} 
		\begin{split} 
			\ell_R((I^{n+1}M &:_M\m)/I^{n+1} M ) \leq \ell_{R\ltimes M}(J^{n+1}:_{R\ltimes M} (\m \times M)/J^{n+1})\\
			 & =\ell_R((I^{n+1}:_R \m)\cap (I^{n+1}M:_R M)/I^{n+1} ) +\ell_R((I^{n+1}M:_M\m)/I^{n+1} M )\\
		 &\leq \ell_R(I^{n+1}:_R \m/I^{n+1} ) +\ell_R((I^{n+1}M:_M\m)/I^{n+1} M ).
		\end{split}
	\end{align*}
This proves the first statement.
Dividing by $n^{t-1}$ and taking $n\rightarrow \infty$, we get that $$\mathrm{f}_I^0(M) \leq \mathrm{f}_J^0(R\ltimes M) \leq \mathrm{f}_I^0(R)+\mathrm{f}_I^0(M) .$$

(i)  $\ir_M(I^{n+1}M) = \ir_{R\ltimes M}(J^{n+1}) $ if and only if $(I^{n+1}:_R \m)\cap (I^{n+1}M:_R M)=I^{n+1}.$ This is equivalent to $(0:_{R/I^{n+1}}\m) \cap (0:_{R/I^{n+1}} M/I^{n+1}M)=0$. The above condition means $\Soc(R/I^{n+1})=0$ or $M/I^{n+1}M$ is a faithful $R/I^{n+1}$-module. Since $R/I^{n+1}$ is an Artinian local ring, the former condition implies that $M=0$. Given that $M$ is nonzero, statement (i) is proven. 

(ii) $\ir_{R\ltimes M}(J^{n+1}) = \ir_R(I^{n+1})+\ir_M(I^{n+1}M)$ if and only if $I^{n+1}:_R \m \subseteq I^{n+1}M:_R~M$. This is equivalent to $(I^{n+1}:_R \m)M\subseteq I^{n+1}M$, which is also equivalent to $(I^{n+1}:_R\m)M= I^{n+1}M$.
\end{proof}

For a finitely generated $R$-module $M$ we call $\dim_{R/\m} \big(\Ext_R^r(R/\m, M)\big)$, where $\depth_R(M)=r$ the {\it type}  of $M$ and is denoted by $\mathrm{r}_R(M)$. If $M$ is Cohen-Macaulay, it is called the {\it Cohen-Macaulay type} of $M$. It is known that the Cohen-Macaulay type, the index of reducibility of $QM$, and the dimension of socle $\Soc(M/Q M)$, where $Q$ is a parameter of a Cohen-Macaulay $R$-module $M$ coincide as in the following remark (see \cite{CQT}).
\begin{remark}\label{lhvb}
	Suppose that $M$ is a Cohen-Macaulay $R$-module and $Q$ is a parameter ideal of $M$. Then
	$$\mathrm{r}_R(M)=\ell_R((Q:_M\m)/Q M)=\ir_M(QM)=\dim_{R/\m}(\Soc(M/Q M))= \mathrm{f}_Q^0(M).$$
	Furthermore, we get that
	\begin{align}\label{ctkq}
	\ir_M(Q^{n+1}M)=\mathrm{f}_Q^0(M)\binom{n+t-1}{t-1}.
	\end{align}
\end{remark}
Then we get one of the main results in \cite{GKL}. Recall that a finitely generated $R$-module $M$ is a maximal Cohen-Macaulay $R$-module if $\depth_R(M)=\dim R$.
\begin{corollary} \cite[Theorem 2.2]{GKL} \label{Goto} Assume $R$ is Cohen-Macaulay. Let $M$ be a maximal Cohen-Macaulay $R$-module. Then
$$
\mathrm{r}_R(M) \leq \mathrm{r}(R \ltimes M) \leq \mathrm{r}(R)+\mathrm{r}_R(M) .
$$
Let $Q$ be a parameter ideal of $R$ and set $\overline{R}=R / Q, \overline{M}=M / Q M$. We then have the following.

(i) $\mathrm{r}(R \ltimes M)=\mathrm{r}_R(M)$ if and only if $\overline{M}$ is a faithful $\overline{R}$-module.

(ii) $\mathrm{r}(R \ltimes M)=\mathrm{r}(R)+\mathrm{r}_R(M)$ if and only if $\left(Q:_R \mathfrak{m}\right) M=Q M$.
\end{corollary}
\begin{proof}
Assume $Q=(r_1, \dots, r_d)$ is a parameter ideal of $R$. Since $M$ is a maximal Cohen-Macaulay $R$-module, its dimension is $d$, and thus $Q$ is also a parameter ideal of $M$. Let $\overline{Q}$ be an ideal of $R \ltimes M$ generated by $(r_1,0), \dots, (r_d,0)$. Then $\overline{Q}=Q\times Q M$ is a parameter ideal of $R \ltimes M$. By Theorem \ref{T2}, 	$$
\mathrm{f}_Q^0(M) \leq \mathrm{f}_{\overline{Q}}^0(R\ltimes M) \leq \mathrm{f}_Q^0(R)+\mathrm{f}_Q^0(M) .
$$
Since $R$ is a Cohen-Macaulay ring and $M$ is a maximal Cohen-Macaulay $R$-module, both $R$ and $M$ are Cohen-Macaulay $R$-modules. Furthermore, the idealization ring $R \ltimes M$ is also Cohen-Macaulay (see \cite[Corollary 4.14]{AW}).
Then we get by Remark \ref{lhvb} that $\mathrm{f}_I^0(M)=\mathrm{r}(M)$ and $\mathrm{f}_{\overline{Q}}^0(R\ltimes M)=\mathrm{r}(R \ltimes M).$ 
	Therefore
$$
\mathrm{r}_R(M) \leq \mathrm{r}(R \ltimes M) \leq \mathrm{r}(R)+\mathrm{r}_R(M) .
$$
This proves the main inequality.

(i) By Remark \ref{lhvb}, $\mathrm{r}_R(M) = \ir_M(QM)$ and $\mathrm{r}(R \ltimes M) = \ir_{R\ltimes M}(\overline{Q}(R\ltimes M))$. Thus, the equality $\mathrm{r}(R \ltimes M)=\mathrm{r}_R(M)$ is equivalent to $\ir_{R\ltimes M}(\overline{Q}(R\ltimes M)) = \ir_M(QM)$. Applying Theorem \ref{T2}(i) with $I=Q$ and $n=0$, this equality holds if and only if $\overline{M}$ is a faithful $\overline{R}$-module.

(ii) We have $\mathrm{r}(R \ltimes M) = \ir_{R\ltimes M}(\overline{Q}(R\ltimes M))$, $\mathrm{r}(R) = \ir_R(Q)$, and $\mathrm{r}_R(M) = \ir_M(QM)$ by Remark \ref{lhvb}.
Therefore, the equality $\mathrm{r}(R \ltimes M)=\mathrm{r}(R)+\mathrm{r}_R(M)$ is equivalent to $\ir_{R\ltimes M}(Q \times QM) = \ir_R(Q) + \ir_M(QM)$.
Applying Theorem \ref{T2}(ii) with $I=Q$ and $n=0$, we directly find that this equality holds if and only if $(Q:_R \m)M= QM$.
\end{proof}
An example for condition (ii) in Corollary \ref{Goto} is as follows. Let $(R,\m)$ be a non-regular Cohen-Macaulay local ring and $i\geq 0$ be an integer. Let $M=\Omega_R^i(R/\m)$ denote the $i$-th syzygy module of the simple $R$-module $R/\m$ in its minimal free resolution. Then by \cite[Theorem 4.1]{GKL}, $\left(Q:_R \mathfrak{m}\right) M=Q M$ for every parameter ideal $Q$ of $R$.

The following example shows that the inequalities in Corollary \ref{Goto} can be strict.
\begin{example}\label{ineq} (\cite[Example 2.3]{GKL}) 
	Let $k$ be a field and $\ell \geq 2$ be an integer. Set $S=k[[X_1, X_2, \ldots, X_{\ell}]]$ the formal power series ring in variables $X_1, X_2, \ldots, X_{\ell}$. Let $\frak{a}=\mathbb{I}_2(\mathbb M)$ denote the ideal of $S$ generated by the maximal minors of the matrix $$\mathbb M= 
	\begin{pmatrix}
		X_1 & X_2 &\ldots & X_{\ell -1} & X_{\ell}\\
	X_2 & X_3 &\ldots & X_{\ell } & X_{1}^q
	\end{pmatrix},
$$
where $\q \geq 2$. We set $R=S/\frak a$. Then $R$ is a Cohen-Macaulay local ring of dimension one. For each integer $2\leq p\leq \ell$, we consider the ideal $I_p=(x_1)+(x_p, x_{p+1}, \ldots, x_{\ell})$ of $R$, where $x_i$ denotes the image of $X_i$ in $R$. Then by \cite[Example 2.3]{GKL},  $r(R)=\ell-1,$  $r(R \ltimes I_p)=(l-p)+1$ and
$$r_R(I_p)=
\begin{cases}
	\ell & \mbox{if} \,\, p=2\\
	\ell -1 & \mbox{if} \,\, p\geq3
\end{cases}
$$
for each $2\leq p\leq \ell.$ Let $\ell=p=3$ and set $M=I_3$. we have $$r(M)<r(R\ltimes M)<r(R)+r(M).$$ Note that $M$ is Cohen-Macaulay $R$-module of dimension 1. Let $Q=(r)$ be a parameter ideal of $R$. Then it is a parameter ideal of $M$. Set $\overline{Q}=(r,0)$. Then $\overline{Q}$ is a parameter ideal of $R \ltimes M$ and $\overline{Q}=Q\times Q M$. By Remark \ref{lhvb}, $\mathrm{f}_Q^0(M)=r_R(M)$. So, 
$$\mathrm{f}_Q^0(M)<\mathrm{f}_{\overline{Q}}^0(R\ltimes M)<\mathrm{f}_Q^0(R)+\mathrm{f}_Q^0(M).$$
Since $R$, $M$ and $R\ltimes M$ are Cohen-Macaulay of dimension 1, by (\ref{ctkq}) in proof of Remark \ref{lhvb}, 
$$\ir_M(Q^{n+1}M)<\ir_{R\ltimes M}(\overline{Q}^{n+1})<\ir_R(Q^{n+1})+\ir_M(Q^{n+1}M).$$
\end{example}

The following example shows that the irreducible multiplicity is not compatible with the reduction ideal. 

\begin{example} \label{VD2} Consider the ring $R$  in Example \ref{ineq}, where $p=\ell =q=2.$ Then
	$$R=k[[X_1,X_2]]/(X_1^3-X_2^2)\cong k[[t^2, t^3]],$$
	where $x_1 \mapsto t^2, x_2 \mapsto t^3$ and $x_i$ denotes the image of $X_i$ in $R$. Set $I=I_2$ and $J=(x_1)$. Then $I$ is the maximal ideal $\m$ of $R$. Since $I^2=JI$, we get that  $J$ is a reduction ideal of $I$. By Example \ref{ineq}, $r(R)=\ell-1=1.$ Since $J$ is a parameter ideal of the Cohen-Macaulay ring $R$, $\mathrm{f}_J^0(R)=1$ by Remark \ref{lhvb}. Now, we compute $\ell_R\big((I^{n+1}  :\m)/I^{n+1} \big)$. Since $I^{n+1}=x_1^n I$, 
		$$I^{n+1}  :\m \cong (t^{2n+2}, t^{2n+3}):(t^2, t^3)=(t^{2n}, t^{2n+1}).$$
		Hence
		\begin{align*}
			\ell_R\big((I^{n+1}  :\m)/I^{n+1} \big) &=\ell_R\big((t^{2n}, t^{2n+1})/(t^{2n+2}, t^{2n+3}) \big)\\
			&=\dim_k(k t^{2n}+kt^{2n+1})\\
			&=2.
		\end{align*}
	So, $\mathrm{f}_I^0(R)=2$ and $\mathrm{f}_I^0(R) \neq \mathrm{f}_J^0(R).$
	\end{example}

{\it Acknowledgment.} The author wishes to express his thanks to S. Kumashiro for useful discussion for this work.

\end{document}